\documentclass{ps}
\usepackage{palatino}
\usepackage{upgreek}
\usepackage{amsfonts, amsmath, amssymb, amsthm}
\usepackage{mathrsfs}
\usepackage{caption}
\usepackage{float}
\usepackage{graphicx}
\usepackage[dvipsnames]{xcolor}
\usepackage{enumerate}
\usepackage{setspace}

\usepackage[square,numbers]{natbib}
\usepackage{hyperref}

\hypersetup{
    colorlinks=true,
    linkcolor=Blue,
    citecolor=Blue,
    urlcolor=Blue,
}

\usepackage{scalerel}

\DeclareMathOperator{\Gr}{\Gamma}
\DeclareMathOperator{\Sgen}{\mathrm{S}}

\DeclareMathOperator{\GHto}{\xrightarrow{\text{GH}}}
\DeclareMathOperator{\N}{\mathbb{N}}
\DeclareMathOperator{\Z}{\mathbb{Z}}
\DeclareMathOperator{\R}{\mathbb{R}}
\DeclareMathOperator{\Green}{\mathsf{G}}
\DeclareMathOperator{\p}{\mathsf{p}}
\DeclareMathOperator{\q}{\mathsf{q}}
\DeclareMathOperator{\g}{\mathfrak{g}}
\DeclareMathOperator{\tor}{\text{tor}}
\DeclareMathOperator{\loz}{\scaleto{\blacklozenge}{4pt}}

\newtheorem*{thm}{Theorem}
\begin{document}
\title{The asymptotic shape theorem for the frog model on finitely generated abelian groups}\thanks{This study was financed in part by the Coordena\c{c}\~ao de Aperfei\c{c}oamento de Pessoal de N\'ivel Superior - Brasil (CAPES) - Finance Code 001.}\thanks{The first author was partially supported by grant \#2017/10555-0, S\~ao Paulo Research Foundation (FAPESP).}
\author{Cristian F. Coletti}\address{Centro de Matem\'atica, Computa\c{c}\~ao e Cogni\c{c}\~ao, Universidade Federal do ABC, 
Av. dos Estados, 5001.
09210-580 Santo Andr\'e, S\~ao Paulo,
Brazil. \email{cristian.coletti@ufabc.edu.br \ \& \ lucas.roberto@ufabc.edu.br}}
\author{Lucas R. de Lima}\sameaddress{1}
%
%
\begin{abstract}
    We study the frog model on Cayley graphs of groups with polynomial growth rate $D \geq 3$. The frog model is an interacting particle system in discrete time. We consider that the process begins with a particle at each vertex of the graph and only one of these particles is active when the process begins. Each activated particle performs a simple random walk in discrete time activating the inactive particles in the visited vertices. We prove that the activation time of particles grows at least linearly and we show that in the abelian case with any finite generator set the set of activated sites has a limiting shape.
\end{abstract}
%
%
\subjclass[2010]{60K35, 60D05, 52A22, 60F15, 60J10}
\keywords{Shape theorem, frog model, Cayley graph, interacting particle system}
\maketitle
\section{Introduction and main result}

We consider the frog model on a Cayley graph $\mathcal{C}(\Gr, \Sgen)$ with polynomial growth rate $D \geq 3$. The model describes an interacting particle system where each particle may be in one of two states, active or inactive. Firstly we introduce the model in a descriptive way and the formal definition will be given in the next subsection. The inactive particles remain in the same place until they become active, which occurs when an active particle visits its site. Once a particle is activated, it starts a simple random walk and does not return to the inactive state. The described process is widely known as \emph{the frog model}, since the particles can be seen as frogs jumping between the neighboring vertices of a graph performing an awakening process. The model can also be interpreted as a rumor transmission model where the active particles are individuals carrying an information which is shared with the inactive particles (see \cite{lebensztayn2013}).

We set the initial configuration of our model with a particle at each vertex of $\mathcal{C} (\Gr, \Sgen)$ where there is only one active particle at time zero, which we can choose to be in $e \in \Gr$, the neutral element of the group $(\Gr,.)$.  In section \ref{sec.activation.time} we proceed with the study of the activation time of the particles in the system and it is shown that it presents at least linear growth. In section \ref{sec.asymp.shape} we consider that $\Gr$ is a finitely generated abelian group to prove the asymptotic shape theorem.

The model has been extensively investigated on the hypercubic lattice $\Z^D$. It was initially studied by \citet{telcs1999} as \emph {the egg model}. Later, \citet*{alves2002} studied the same model under the name of \emph{the frog model} giving the first proof of the shape theorem on $\Z^D$ with $D\geq  2$. To the best of our knowledge, the present work is the first to consider the frog model in a more general algebraic structure, namely, on groups of polynomial growth for any symmetric and finite generator set. It covers the case $\Z^D $ with $D\geq 3$. We restrict our attention to obtain a shape theorem in the abelian case. There are also other variations of the model on other structures such as the frog model on trees \cite{hoffman2017, lebensztayn2005}, in the continuum \cite{beckman2018} and, more generally, on any discrete set associated with a set of paths \cite{kosygina2017}. We can also find in the literature other shape theorems for some variations of the frog model. For instance, \citet{ramirez2004} studied a theorem for the continuous-time model on $\Z^D$ while an analogous result for the frog model on trees was studied by \citet{hoffman2019}

\subsection{Description of the model}

To provide a formal definition of the model, we first introduce the structure and, subsequently, the random variables that characterize the process.

\subsubsection{Cayley graphs of polynomial growth}

Let $(\Gr,.)$ be a group finitely generated by a symmetric $\Sgen \subseteq  \Gr$, \emph{i.e.}, if $s \in \Sgen$, then $s^{-1} \in \Sgen$. We define the \emph{undirected Cayley graph} associated to $\Gr$ and $\mathrm{S}$ by $\mathcal{C}(\Gr, \mathrm{S}) = (V,E)$, where $V=\Gr$ is the set of vertices and $E = \big\lbrace\{sx,x \} : s \in \mathrm{S}, x \in \Gr \big\rbrace$ is the set of edges. In our model we consider that $\Sgen$ is finite and $e\not\in\Sgen$, where $e$ is the neutral element of $\Gr$ (with no loops).

Let us denote by $\mathscr{P}(x,y)$ the set of finite paths from $x$ to $y$ and let each $p \in \mathscr{P}(x,y)$ be given by $p=(x=x_0, x_1, \dots, x_{m-1},x_m = y)$ with $\{x_{i-1},x_{i}\} \in E$. A \emph{word metric} $d_{\mathrm{S}}$ associated to $\mathcal{C}(\Gr,\Sgen)$ is a metric such that, for all $x,y \in \Gr$,
\[
    d_S(x, y) = \inf\left\lbrace \sum\limits_{i=1}^{|p|} w_S(x_{i-1}, x_{i}): p \in \mathscr{P}(x,y)\right\rbrace,
\]
where $w_S: E \to \R_+^*$ defines a weight on the edges of $\mathcal{C}(\Gr,\Sgen)$ with $w_S(z,sz) =w_S(z,s^{-1}z)$ for every $z\in \Gamma$ and $s\in S$.

We denote by $d$ the word metric such that $d(sx,x) = 1$ for all $x \in \Gr$ and $s \in \Sgen$. We associate to $d$ a function $\|\cdot\|_1 : \Gr \to \R_+$ given by $\|x\|_1 = d(e,x)$ for every $x \in \Gr$.

Let $B(x, r) = \{ y \in \Gr : d(x,y) \leq r \}$ be the \emph{ball with radius $r$ centered at $x$} given by the metric $d$ and, more generally, let $B_f(x, r) = \{ y \in X : d_f(x,y) \leq r \}$ be the \emph{ball centered at $x \in X$ with radius $r$} associated to a pseudometric $d_f$ on $X$. We say that $\mathcal{C}(\Gr,\Sgen)$ has \emph{polynomial growth} if there exist $D' \in \N^*$ and a constant $C>0$ such that
\[
    |B(x, r)| \leq C r^{D'}
\]
for any $x\in \Gr$ and $r\in \N^*$. In fact, by a theorem of \citet{bass1972}, if a Cayley graph $\mathcal{C}(\Gr,\Sgen)$ has polynomial growth, then there exist $c>0$ and $D\in\N$ such that
\[
    \frac{1}{c}r^D \leq |B(x, r)| \leq c r^D.
\]
We call such $D$ the \emph{polynomial growth rate} of $\mathcal{C}(\Gr,\Sgen)$.

Let $[x,y] := xyx^{-1}y^{-1}$ denote the commutator element of $x$ and $y$ in $\Gr$ and define\break$[H_1,H_2] := \big\langle [h_1,h_2] : h_1 \in H_1, h_2\in H_2 \big\rangle$ where  $H_1,H_2 \subseteq \Gr$. Let $(C^n(\Gr))_{n\in\N}$ be a decreasing sequence of subgroups given by $C^0(\Gr) = \Gr$ and $C^{i+1}(\Gr)=[\Gr,C^i(\Gr)]$. We say that a group $N$ is \emph{nilpotent} of degree $k$ if there exists $k = \inf \big\lbrace n \in \N : C^n(N)=\{e\} \big\rbrace < \infty$. The group $\Gr$ is called \emph{virtually nilpotent} it there exists a $N \unlhd \Gr$ nilpotent such that $[\Gr:N] < \infty$.

When considering $\Gr$ finitely generated by $\Sgen$ such that $\mathcal{C}(\Gr,\Sgen)$ has polynomial growth, it follows from a theorem of \citet[p.~54]{gromov1981} that $\Gr$ is virtually nilpotent. By the structural theorem for finitely generated abelian groups, if $\Gr$ is abelian and finitely generated by $\Sgen$, then $\mathcal{C}(\Gr,\Sgen)$ has polynomial growth rate $D$ and $\Gr$ is isomorphic to an additive group 
\begin{equation} \label{structural.abelian.thm}
    \Gr \cong \Z^D \oplus \Z_{m_1} \oplus \Z_{m_2} \oplus \cdots \oplus \Z_{m_\ell}
\end{equation}
with $m_i \in \N$ for all $i \in \{1,2, \dots , \ell\}$ and $m_j ~|~ m_{j+1}$ for $j\in \{1,2, \dots , \ell-1\}$. It is worth pointing out that even though $\Gr$ is isomorphic to $\Z^D$ with a torsion subgroup, the corresponding Cayley graphs of either $\Gr$ or $\Gr/\tor\Gr \cong \Z^D$ are not necessarily isomorphic to the nearest neighbor edge (hypercubic) $\Z^D$ lattice.

\subsubsection{The frog model}
Consider the Cayley graph $\mathcal{C}(\Gr,\Sgen)$ with polynomial growth rate $D \geq 3$. Associate to each $x \in \Gr$ the probability space $(\Omega_x,\mathscr{F}_x,\mathbb{P}_x)$ with $\Omega_x := \{x\} \times \Sgen^{\N^*}$. We define the \emph{simple random walk starting at} $x$ as the sequence of random elements $\big(\check{S}_n^x\big)_{n\in\mathbb{N}}$ relative to the position of the particle at time $n$ where $\check{S}_0^x := x$ and, for a given $\check{\omega} = (x,(\xi_i)_{i \in \N^*}) \in \Omega_x$,
\[
    \Check{S}^x_n (\check{\omega}) = \xi_n\xi_{n-1} \dots \xi_2\xi_1x.
\]
Here $\mathbb{P}_x \left( \Check{S}^x_{n+1} = s\Check{S}^x_{n} \right) = 1/|\mathrm{S}|$ for every $s \in \mathrm{S}$ and $n\in\N$. We consider on $\Omega_x$ the $\sigma$-algebra\break$\mathscr{F}_x := \sigma \big( \check{S}_n^x : n \in \N \big)$. Now, we define the probability space for the frog model as $(\Omega, \mathscr{F},\mathbb{P})$ where $\Omega = \prod\limits_{x \in \Gr} \Omega_x$, $\mathscr{F}$ is the product $\sigma$-algebra and $\mathbb{P}$ is the corresponding product probability measure. For each $x \in \Gr$, let $\pi_x: \Omega \to \Omega_x$ be the projection such that for a given $\omega \in \Omega$, we have that $\omega = (\pi_x(\omega))_{x\in \Gr}$.

The family of sequences of random elements $\big\lbrace(S_n^x)_{n\in\mathbb{N}} : x \in \Gr\big\rbrace$ given by $S^x_n := \check{S}_n^x \circ \pi_x$ represent the independent simple random walks of every particle on the graph.

For $x,y \in \Gr$ define a random variable $t(x,y) = \inf\{ n\in \mathbb{N} : S_n^x = y \}$. Note that $t(x,y) = \infty$ with positive probability since the random walks on $\Gr$ are transient whenever its polynomial growth rate is at least $3$ \cite[see][p.~R59]{burioni2005}. The time when the particle with initial position $y$ becomes active in the process starting with the one active particle at $x$ is given by the random variable
\[T(x,y) = \inf\left\lbrace \sum_{i=1}^m t(x_{i-1},x_i) : x_0=x, x_1, \dots , x_m = y \right\rbrace.\]

Observe that $T$ is  not necessarily symmetric; also, in a further section, we will consider a random quasimetric $d_{\omega}$ with $d_{\omega}(e,x) = T(e,x)(\omega) ~\mbox{a.s.}$ The random set $B_{\omega}(e,n)$ corresponds to the set of the original positions of the active particles up to time $n$ starting from one active particle at $e$. In particular, the shape theorem studied in this paper refers to the behavior of this random set.

\subsection{On the convergence of metric spaces} \label{conv.metric.spaces}

Due to the abstract generality of the spaces in which we define the frog model, we will use some concepts about the convergence of metric spaces to study the growth of the random sets related to our process.

Given a metric space $(X,d_X)$ and non-empty subsets $A,B \subseteq X$, a \emph{$\varepsilon$-neighborhood} of $A$ is the set\break$[A]_\varepsilon := \{ x \in X : \exists a \in A (d_X(a,x)  < \varepsilon) \}$ and the \emph{Hausdorff distance} between $A$ and $B$ is given by
\[d_H(A,B) = \inf \big\lbrace \varepsilon > 0 : A \subseteq [B]_\varepsilon \text{ and } B \subseteq [A]_\varepsilon \big\rbrace.\]

Let $\left( (X_n,d_n) \right)_{n\in\N^*}$ be a sequence of compact metric spaces with uniformly bounded diameter. We say that $(X_n,d_n)$ converges to a compact metric space $(X,d_{\tilde{X}})$ \big(subspace of $(\tilde{X}, d_{\tilde{X}})$\big) in the \emph{Gromov-Hausdorff} sense if there exist isometric embeddings $\varphi_n: X_n \hookrightarrow \tilde{X}$ such that $\lim\limits_{n \to \infty}d_H(\varphi_n[X_n],X) = 0$. We denote such convergence by
\[(X_n,d_n) \GHto (X,d_X).\]

More generally, consider now a sequence $\left( (X_n,d_n,o_n) \right)_{n\in\N^*}$ where every $(X_n,d_n)$ is a locally compact metric space and $o_n \in X_n$. The \emph{pointed Gromov-Hausdorff convergence} of $(X_n,d_n,o_n)$ to $(X,d_X,o)$ occurs if, for each $r > 0$, we have that $(B_n(o_n, r), d_n) \GHto (B_X(o, r), d_X)$ and we denote this convergence by
\[(X_n,d_n,o_n) \GHto (X,d_X,o).\]
The limit object given above is also known as the \emph{asymptotic cone} of $(X_n,d_n,o_n)$.

We will state the shape theorem for the frog model in terms of pointed Gromov-Hausdorff convergence. First we present a deterministic result on the convergence of rescaled word metrics due to \citet{pansu1983} and describe the construction of the limit space given by \citet{cantrell2017}.

\begin{thm}[{\citet{pansu1983}}]
    Let $\Gr$ be a virtually nilpotent group generated by a symmetric and finite $\Sgen \subseteq \Gr$. Then
    \[
        \left( \Gr, \frac{1}{n}d_S, e \right) \GHto (G_\infty, d_\infty, e),
    \]
    where $G_\infty$ is a simply connected real graded Lie group and $d_\infty$ is a right-invariant sub-Riemannian (Carnot-Caratheodory) metric homogeneous with respect to a family of homotheties $\{\delta_t\}_{t>0}$, \emph{i.e.}, $d_\infty(\delta_t(g),\delta_t(h)) = t~d_\infty(g,h)$ for all $t>0$ and $g,h \in G_\infty$. 
\end{thm}

The construction of the limit space can be briefly described by taking $N \unlhd \Gr$ nilpotent with $[\Gr:N] < \infty$. Then we define a nilpotent and torsion-free group $\Gr' = N /\tor N$ where  $\tor N := \big\langle x \in N : \exists n\in\N^* (x^n = e)  \big\rangle$ is the \emph{the torsion subgroup} with  $\tor N \unlhd N$ finite \cite[see][\S A]{pansu1983}. We consider $G$ the real and simply connected Lie group given by the Mal'cev completion of $\Gr'$. Therefore $\Gr'$ is cocompact in $G$ and the Hausdorff distance between $\Gr$ and $\Gr'$ is finite \cite[see][p.~434]{pansu1983}. Considering the rescaled quotient metric, we can verify that $\Gr$ converges to $G_\infty$, as in the construction present in \cite[\S 2.1]{cantrell2017}.

When $\Gr$ is abelian and finitely generated, we can verify by \eqref{structural.abelian.thm} that $\Gr' = \Gr/\tor\Gr \cong \Z^D$ and $G = G_\infty \cong \R^D$ where $G_\infty$ is a Riemannian manifold when associated with the metric $d_\infty$.

\subsection{The asymptotic shape theorem}

Shape theorems are commonly studied for first passage percolation and other random growth models (see \cite{alves2002,auffinger2017,hoffman2019,ramirez2004}, for instance). The statement of the theorem describes the behavior of a random set which grows with time. It can be seen as an analogue of the Strong Law of Large Numbers for processes on graphs. Roughly speaking, this set coincides with the balls of a random pseudo-quasi metric and we seek to describe to which set it converges and the properties it has.

\begin{thrm}[Asymptotic shape theorem for the frog model on  $\mathcal{C}(\Gr,\Sgen)$] \label{shape.thm}
    Let $(\Gr,.)$ be an abelian group generated by a symmetric and finite $\Sgen \subseteq \Gr$ such that $\mathcal{C}(\Gr,\Sgen)$ has polynomial growth rate $D \geq 3$ and no loops. Consider the frog model defined on $\mathcal{C}(\Gr,\Sgen)$. Then there exists $\Omega' \subseteq \Omega$ such that $\mathbb{P}(\Omega') =1$ and, for all $\omega \in \Omega'$, 
    \[
        \left( \Gr, \frac{1}{n}d_{\omega}, e \right) \GHto \left( G_\infty, d_\phi, e \right)
    \]
     where $d_\omega(e,x) := T(e, x)(\omega)$, $G_\infty \cong \R^D$, $d_\phi$ is a right-invariant metric on $G_\infty$, not necessarily symmetric, homogeneous with respect to a family of homotheties $\{\delta_t\}_{t>0}$, and bi-Lipschitz with respect to a Riemannian metric on $G_\infty$.
\end{thrm}

The need to state the theorem in terms of a centered Gromov-Hausdorff convergence is a consequence of the fact that $\Gr$ may fail to be isomorphic to a subgroup of $G_\infty$, as can be seen in \eqref{structural.abelian.thm} and in subsection \ref{conv.metric.spaces}. If $\Gr \cong \Z^D$ the shape theorem could be stated in the same way as the classical one, where there exists $\Omega' \subseteq \Omega$ with $\mathbb{P}(\Omega') =1$ such that, given $\omega \in \Omega'$ and $\varepsilon>0$, there exists $n_0 \in \N$ such that, if $n \geq n_0$,
\[
    B_{\phi}(0, nr(1-\varepsilon))\cap \Gr \subseteq B_{\omega}(0, nr) \subseteq B_{\phi}(0, nr(1+\varepsilon)).
\]

In particular, the Hausdorff distance between $\delta_{1/n}[B_\omega(e,nr)]$ and $B_{\phi}(e,r)$ tends to $0$ for all $r>0$, where $\delta_{t'}$ is a homothety on $G_\infty \cong \R^D$ and $d_\phi$ is bi-Lipschitz with respect to $d_\infty$. We will use a subadditive ergodic theorem to prove Theorem \ref{shape.thm}. The condition $D \geq 3$ in that theorem is a consequence of the application of some known results for random walks on groups (see \S\ref{sec.activation.time}). Therefore, the case $D \leq 2$ should be treated separately and we do not aim to find a sharp lower bound for $D$ in our shape theorem.

\section{On random walks and activation times} \label{sec.activation.time}

To obtain results for the random variable $T(\cdot,\cdot)$ we adopt a similar approach to that of \citet*{alves2002} generalizing some of their results. Consider the \emph{heat kernel} of the random walk starting at $x$ and arriving at $y$ given by
\[
    \p_n(x,y) = \mathbb{P}(S_n^x=y).
\]

We define the \emph{Green's function} as the mean number of visits from a particle starting at $x$ visiting $y$ up to time $n$ 
\begin{equation} \label{Green.def}
    \Green_n(x,y) = \sum\limits_{i=0}^n \p_i(x,y)
\end{equation}
and $\Green(x,y)= \lim\limits_{n\to\infty}\Green_n(x,y)$. Let the probability of a simple random walk starting at $x$ reach site $y$ up to time $n$ be given by $\q_x(n,y) = \mathbb{P}\left(t(x,y) \leq n\right)$.
\begin{prpstn}[{\citet[p.~731]{alexopoulos2002}}] \label{heat.kernel.pn}
    There exists a constant $c >0$ such that, for all $n \in \mathbb{N}^*$,
    \begin{equation*}
        \label{upper.bound.heat.kernel}
        \p_n(x,y) \leq c n^{-D/2}\exp{\left( -\frac{d(x,y)^2}{c n} \right)}.
    \end{equation*}
    Moreover, if the graph is not bipartite, then there exists a constant $c' > 0$, such that, for all $n \in \mathbb{N}^*$,
    \begin{equation}
        \label{lower.bound.heat.kernel}
        \p_n(x,y) \geq \frac{1}{c'}n^{-D/2}\exp{\left( -c'\frac{d(x,y)^2}{n} \right)}
    \end{equation}
     whenever $d(x,y) \geq n/c'$. For combinatorial reasons, \eqref{lower.bound.heat.kernel} holds for bipartite graphs when $n$ has the same parity of $d(x,y)$.
\end{prpstn}

\begin{prpstn} \label{prop.qx.lim.inf}
    Let $\mathcal{C}(\Gr,\Sgen)$ have polynomial growth rate $D \geq 3$. Then there exists a constant $C>0$ such that
    \[
        \q_x(n,y) \geq \frac{C}{d(x,y)^{D-2}}
    \]
    for all $n \geq \max\{d(x,y)^2,c'd(x,y)\}$, where $c'$ is given by \eqref{lower.bound.heat.kernel}.
\end{prpstn}

\begin{proof}
    We first observe that if $n' > n$, then $\q_x(n', y) \geq \q_x(n, y)$. Therefore we may consider w.l.o.g. that $n = \left\lceil\max \{d(x,y)^2, c'd(x,y)\}\right\rceil$.

    We follow an analogous procedure to the one adopted in Theorem 2.2 of \cite{alves2002}. We get from \eqref{Green.def} that
    \begin{eqnarray*}
    \Green_n(x,y) & \leq & \sum_{j=0}^n\sum_{k=0}^j\p_k(x,x)\mathbb{P} \big( t(x,y)=j-k \big) \\
    & = & \sum_{k=0}^n\p_k(x,x)\q_x(n-k,y) \leq \q_x(n,y)\Green_n(x,x).
    \end{eqnarray*}

    In particular, by Proposition \ref{heat.kernel.pn}, $\Green(x,x)$ converges. Hence $\q_x(n,y) \geq \dfrac{\Green_n(x,y)}{\Green(x,x)}$. Since
    \[
        \Green_n(x,y) \geq c_0 \sum_{j= \lfloor n/4 \rfloor}^{\lfloor n/2 \rfloor}(2j+1)^{-D/2} \geq c_0'n^{-D/2},
    \]
   the result follows from the fact that if $c' <1$ then $n=d(x,y)^2$, and $n \leq c'd(x,y)^2$ for $c' \geq 1$. \qedhere
\end{proof}

We denote by
\[
    \uptau^x_{r} := \inf \big\lbrace n \in \N : d(x,S_n^x) > r \big\rbrace
\]
 the \emph{first exit of $(S_n^x)_{n \in \N^*}$ from the ball $B(x,r)$} for which we present the following result.

\begin{prpstn} \label{prop.sup.rw.ball}
    Let $\mathcal{C}(\Gr,\Sgen)$ have polynomial growth rate $D \geq 3$. Then there exist constants $c_1,c_2 > 0$ such that, given $n \in \N^*$, $t \in \R^*_+$ and $x\in \Gr$,
    \[
        \mathbb{P}\left(\uptau^x_{t\sqrt{n}} \leq n \right) \leq c_1\exp\left(- c_2{t^2}\right).
    \]
\end{prpstn}
\begin{proof}
    It suffices to apply Lemma 12.3 from \citet{alexopoulos2002} and verify that there exists $c'>0$ such that
    \[
        \mathbb{P}\left(\uptau^x_{t\sqrt{n}} \leq n \right) \leq c'\exp\left(- \frac{\lfloor t\sqrt{n}\rfloor^2}{c'n}\right) \leq 3c'\exp\left(- \frac{t^2}{2c'}\right).
    \]
\end{proof}

We denote by $\mathsf{R}_n^x = \left\lbrace S_i^x : i \in\{0,1, \dots,n\}\right\rbrace$ the random set of \emph{distinct visited sites} of the random walk starting from $x$ up to time $n$ for which we state the following lemma.

\begin{lmm}[{\citet[p.~R59]{burioni2005}}] \label{lm.avrg.distinc.sites}
    Let $\mathcal{C}(\Gr,\Sgen)$ have no loops and polynomial growth rate $D\geq 3$. Then there exists $a_0 >0$ such that
    \begin{equation*}
        \lim_{n \to \infty}\frac{\mathbb{E}\big[|\mathsf{R}_n^x|\big]}{n}=a_0.
    \end{equation*}
\end{lmm}

\begin{rmrk}
    Concerning Lemma \ref{lm.avrg.distinc.sites} stated above, we observe that the case where the graph has polynomial growth rate $D = 2$ is rather particular. For instance, random walks on Cayley graphs with polynomial growth rate 2 are recurrent (see \cite{woess2000}, \S3.B).  In particular, if the graph is the $\Z^2$ lattice, then it can be proved that $\mathbb{E}\big[|\mathsf{R}_n^x|\big]\log(n)/n$ converges.
\end{rmrk}

We can now proceed with the study of the activation times. We begin by proving the following result.

\begin{prpstn} \label{prop.T.finite.as}
    Let $\mathcal{C}(\Gr,\Sgen)$ have no loops and polynomial growth rate $D \geq 3$. Then there exists a constant $\beta >0$ such that, for given $x, x_0 \in \Gr$, there exists $C= C(x_0x^{-1}) >0$ satisfying
    \[\mathbb{P}\big(T(x,x_0)\geq n\big) \leq C \exp{\left(-n^\beta\right)}.\]
\end{prpstn}

\begin{proof}
    The proof consists in following the ideas introduced in Theorem 3.2 of \cite{alves2002}. For the convenience of the reader, we repeat the reasoning with the corresponding adjustments.
    
    Let $n \in \N$ be such that $n \geq  \lceil \max\{ d(x,x_0)^2, c' d(x,x_0) \} \rceil$ where $c'>0$ is given by (\ref{lower.bound.heat.kernel}). Set
    \[
        \mathsf{D}_{i,\upepsilon} := \left\{ y \in \Gr : \|yx^{-1}\|_1 \leq i n^{1/2 + \upepsilon} \right\}
    \]
    with $i \in \{1, \dots, \lfloor D/2 \rfloor\}$ and $\upepsilon \in {}]0,1[$ to be defined later. Define the event
    \[
        A_1 = A_1(n,\upepsilon) := \left\{\left| \mathsf{R}_n^e \cap \mathsf{D}_{1,\upepsilon} \right| \geq r_1 n^{1-\upepsilon}\right\},
    \]
    where $r_1 > 0$ is a constant depending on $D$ which will be chosen later. We continue below with some auxiliary results.
    
    \begin{clm} \label{lm.desig.X}
        Let $X$ be an integer-valued random variable such that $0 \leq X \leq \mathsf{a} ~a.s.$ and $\mathbb{E}[X] \geq \mathsf{b}$ with $\mathsf{b} >0$. Then
        \[
            \mathbb{P}\left(X \geq \frac{\mathsf{b}}{2}\right) \geq \frac{\mathsf{b}}{2\mathsf{a}}.
        \]
    \end{clm}
    \begin{proof} \renewcommand\qedsymbol{$\blacksquare$}
        An easy computation shows that
        \[
            \mathbb{E}[X] = \sum_{j = 1}^{\lfloor \mathsf{b}/2 \rfloor} i \mathbb{P}(X=1) +  \sum_{j \geq \lfloor \mathsf{b}/2 \rfloor +1}^{\lfloor \mathsf{a} \rfloor} i \mathbb{P}(X=1) \leq \frac{\mathsf{b}}{2} + \mathsf{a}\mathbb{P}\left(X \geq \frac{\mathsf{b}}{2} \right).
        \]
        We get to the desired conclusion observing that $\mathbb{E}[X] \geq \mathsf{b}$.
    \end{proof}
    
    \begin{clm} \label{Af.1}
        Let $D \geq 3$. Then one can choose $r_1 >0$ such that there exist $\alpha_1,\alpha_1' >0$ satisfying
        \[
            \mathbb{P}(A_1) \geq 1 - \alpha_1 \exp(-\alpha_1' n^\upepsilon) \quad \text{for all}n \in \N^*
        \]
    \end{clm}
    \begin{proof}\renewcommand\qedsymbol{$\blacksquare$}
        By Lemma \ref{lm.avrg.distinc.sites} and Claim \ref{lm.desig.X}, there exist $r_1, C_1 >0$ such that
        \begin{equation} \label{eq.range}
            \mathbb{P}\left(\left|\mathsf{R}_k^x\right| \geq r_1 k\right) \geq C_1.
        \end{equation}
        Fix
        \[
            A_1' = A_1'(n, \upepsilon) := \left\{ \left| \mathsf{R}_n^x\right| \geq r_1 n^{1-\upepsilon} \right\}.
        \]
        Consider a partition of $[0,n]$ into disjoint intervals of length $n^{1-\upepsilon}$. The cardinality of $|\mathsf{R}_{k}^{x}|$  associated with each subinterval does no depend on the cardinalities of the other subranges. We thus apply (\ref{eq.range}) with $k = n^{1-\upepsilon}$ obtaining
        \begin{equation} \label{eq.A1linha.l11}
            \mathbb{P}(A_1') \geq 1 - (1-C_1)^{n^\upepsilon}.
        \end{equation}
        It follows from Proposition \ref{prop.sup.rw.ball} that there exist $c_1,c_2 >0$ such that
        \begin{equation} \label{eq.B.l11}
            \mathbb{P}\left(\uptau_{n^{1/2 + \upepsilon}}^x \leq n\right) \leq c_1 \exp\left( -c_2 n^\upepsilon \right)  
        \end{equation}
        We verify the claim combining (\ref{eq.A1linha.l11}) and (\ref{eq.B.l11}).
    \end{proof}
    Set $n_k>0$ to be given by
    \[
    n_k = n + \left(\sum_{j=2}^k (2j+1)^2\right) n^{1+2\upepsilon}
    \]
    Define the following random sets
    \begin{equation*}
    \Tilde{G}_1 = \{ y \in \mathsf{D}_{1,\upepsilon}: t(x,y) \leq n_1\},
    \end{equation*}
    and
    \begin{equation*}
    \Tilde{G}_k = \Big\{ y \in \mathsf{D}_{k,\upepsilon}\setminus\mathsf{D}_{k-1,\upepsilon}: \exists z \in \Tilde{G}_{k-1} \Big( t(z,y) \leq n_k - n_{k-1} \Big) \Big\}
    \end{equation*}
    where $k \in \{2, \dots, \lfloor D/2\rfloor\}$. Let us write, for $k \in \{2, \dots, \lfloor D/2\rfloor\}$,
    \[
        A_k = A_k(n,\upepsilon) := \left\{ \left| \Tilde{G}_k \right| \geq r_k n^k  \right\}
    \]
    where each $r_k$ will be chosen later. Now, set
    \[
        \upepsilon(k) = \left\lbrace\begin{array}{cc}
            \upepsilon/2, & \mbox{ if } k=1, \\
            \upepsilon, & \mbox{ if } k>1.
        \end{array} \right.
    \]
    We sate without proof the claim below, which follows in the same lines of Lemma 3.3 of \cite{alves2002} and uses Proposition \ref{prop.qx.lim.inf} and Claim \ref{Af.1}.
    \begin{clm} \label{Af.2}
        Let $D \geq 4$. Then one can choose $r_i >0$ for all $i \in \{2, \dots, \lfloor D/2 \rfloor\}$ such that, for every $k \in \{1, 2, \dots, \lfloor D/2 \rfloor\}$, there exist $\alpha_k,\alpha_k' >0$ satisfying
        \begin{equation*} \label{eq.af.ad1}
            \mathbb{P}(A_{k+1}|A_{k}) \geq 1 - \alpha_k \exp\left(-\alpha_k'n^{2\upepsilon(k)}\right)
        \end{equation*}
        for all $n \in \N^*$. Moreover, if $D \geq 3$, then there exist $\hat{\alpha}_0,\hat{\alpha}_1, \gamma_1>0$ such that
        \begin{equation} \label{eq.af.ad2}
            \mathbb{P}\left(A_{\lfloor D/2 \rfloor}\right) \geq 1 - \hat{\alpha}_0 \exp \left( \hat{\alpha}_1 n^{\gamma_1} \right)
        \end{equation}
        for all $n \in \N^*$.
        \hfill $\blacksquare$
    \end{clm}
    
    We are now in a position to show the case $D \geq 4$. Let us define
    \begin{equation} \label{H}
        H :=\left\{\forall y \in\Tilde{G}_{\lfloor D/2 \rfloor}\left( T(y,x_0)> n_{\lfloor D/2 \rfloor} + (\lfloor D/2 \rfloor +1)^2n^{1+2\upepsilon}\right)\right\}.
    \end{equation}
    Note that when $\Tilde{G}_{\lfloor D/2 \rfloor}$ is conditioned to $A_{\lfloor D/2 \rfloor}$ one has that $|\Tilde{G}_{\lfloor D/2 \rfloor}| \geq r_{\lfloor D/2 \rfloor}n^{\lfloor D/2 \rfloor}$ and $d(x_0, y) \leq \big({\lfloor D/2 \rfloor}+1\big) n^{1/2 + \upepsilon}$ for all $y \in \Tilde{G}_{\lfloor D/2 \rfloor}$. By the independence of the random walks and by Proposition \ref{prop.qx.lim.inf}, there exists $C'> 0$ such that
    \begin{align}
        \mathbb{P}\Big( T(x,x_0) > n_{\lfloor D/2 \rfloor} + ({\lfloor D/2 \rfloor} + 1)^2n^{1+ 2\upepsilon} | A_{\lfloor D/2 \rfloor}\big) &\leq \mathbb{P}(H | A_{\lfloor D/2 \rfloor}\Big) \label{prob.H}\\
        &\leq \left( 1 - \dfrac{C'}{n^{(1/2 + \upepsilon)(D-2)}}\right)^{r_{\lfloor D/2 \rfloor}n^{\lfloor D/2 \rfloor}}. \nonumber
    \end{align}
    Let us now choose $\upepsilon < \frac{1}{2(D-2)}$. Since
    \begin{align*}
        \mathbb{P}\Big( T(x,x_0) > n_{\lfloor D/2 \rfloor} + &({\lfloor D/2 \rfloor} + 1)^2n^{1+ 2\upepsilon}\Big) \\ &\leq \mathbb{P}\left( T(x,x_0) > n_{\lfloor D/2 \rfloor} + ({\lfloor D/2 \rfloor} + 1)^2n^{1+ 2\upepsilon} | A_{\lfloor D/2 \rfloor}\right) + \mathbb{P}\left(A_{\lfloor D/2 \rfloor}^c\right)
    \end{align*}
    we get the result by combining (\ref{eq.af.ad2}) and \eqref{prob.H}.
    
    We now turn to the case $D=3$. Define $H$ as \eqref{H} and we draw the same conclusion as above for the existence of $C''>0$ such that
    \[
        \mathbb{P}(H | A_1 ) \leq \left( 1 -\frac{C''}{n^{1/2+\upepsilon}} \right)^{r_1n^{1-\upepsilon}}.
    \]
    We conclude the proof by applying Claim \ref{Af.1} with $\upepsilon < 1/4$.
\end{proof}

The remainder of this section will be devoted to state and prove that $T(e,\cdot)$ grows at least linearly. The following contruction is adapted from \cite{alves2002}.

\begin{prpstn} \label{prop.T.div.dist}
    Let $\mathcal{C}(\Gr,\Sgen)$ have no loops and polynomial growth rate $D \geq 3$. Then there exist $C>0$,  $\kappa>0$ and $\upalpha >1$ such that
    \[
        \mathbb{P}\left(T(e,x)\geq \upalpha n\right) \leq C \exp{\left(-n^\kappa\right)}
    \]
    for all $x \in \Gr$ and every $n \in \N$ such that $n \geq \|x\|_1$.
\end{prpstn}
\begin{proof}
    The proof will be divided into two parts. We first consider the case $\|x\|_1=n$. Let $p \in \mathscr{P}(e,x)$ be a $d$-geodesic in $\mathcal{C}(\Gr,\Sgen)$. We fix $p=(e = x_0, x_1, \dots, x_n = x)$. Note that $\|x_k\|_1 = k$ for all $k \in \{0, 1, \dots, n\}$. Set $Y_i := T(x_{i-1},x_i)$. Since $T(\cdot,\cdot)$ is subadditive, it suffices to verify the existence of $\upalpha>1$ and $\kappa>0$ satisfying
    \begin{equation*}
        \mathbb{P}\left( \sum_{i=1}^n Y_i \geq \upalpha{n} \right) \leq C \exp(-n^\kappa).
    \end{equation*}
    Set
    \[
        B := \left\{ Y_i \leq \frac{\sqrt{n}}{2} : i \in \{1, 2, \dots, n\}\right\}.
    \]
    It follows from Proposition \ref{prop.T.finite.as} that there exists $C_s$ depending on each $s \in S$ such that 
    \[
        \mathbb{P}(T(e,s) \geq t) \leq C_s\exp(-t^\beta),
    \]
    for all $t>0$. Let $C_S:= \max\{C_s: s \in S\}$. Then there exists $\kappa' >0$ such that
    \begin{equation} \label{eq.lm4.1}
        \mathbb{P}(B^c) \leq C_S n\exp(-n^{\kappa'}).
    \end{equation}
    Let us define
    \[
        \sigma_i := \sum_{j=0}^{M_i} Y_{i+j\lceil \sqrt{n} \rceil}
    \]
    for $i \in \{ 1, 2, \dots, \lceil \sqrt{n} \rceil \}$ where $M_i := \max\{l \in \N : i+l\lceil\sqrt{n}\rceil \leq n\}$.
    
    Note that $\left\{ Y_{i+ j\lceil\sqrt{n}\rceil} : j \in \{1, \dots, M_i  \} \right\}$ is a set of independent random variables when it is conditioned to $B$. Thus $\sigma_i$ is a sum of independent random variables when the event $B$ occurs. Hence the following claim is an immediate consequence of Theorem 1 of \citet{fuk1971}.
    
    \begin{clm}\label{Af.lm4}
        Let $\sigma_i$ be defined as above. Then, for all $\lambda>0$, one has that
        \[
            \mathbb{P}\left( \left.\sigma_i \geq \lambda{M_i}\right| B \right) \leq \exp\left(2\sqrt{\lambda{M_i}}\left(1-\log\left(\frac{\beta\lambda}{C_S~\Upgamma(1/\beta)}+1\right)\right) \right),
        \]
        where $\Upgamma(\cdot)$ is the gamma function.
    \end{clm}
    
    By Claim \ref{Af.lm4}, there exists a sufficiently small $\Hat{\upalpha}>1$ and $C_1>0$ for a $\kappa''>0$ such that
    \begin{eqnarray}
        \mathbb{P}\left(\left. \sum_{i=1}^n Y_i \leq \Hat{\upalpha}{n}\right|B\right) &\leq& \mathbb{P}\left(\left. \left\{ \forall i \in \{1, \dots, \lceil \sqrt{n} \rceil\}\left(\sigma_i\leq \Hat{\upalpha}{M_i} \right)\right\}^c\right|B\right) \nonumber\\
        &\leq& \sum_{i=1}^{\lceil\sqrt{n}\rceil}\mathbb{P}\left(\left.\sigma_i \geq \Hat{\upalpha}{M_i}\right|B\right) \label{eq.lm4.2}\\
        &\leq& C_1 \sum_{i=1}^{\lceil\sqrt{n}\rceil} \exp\left( -M_i^{\kappa''} \right) \nonumber
    \end{eqnarray}
    where $M_i= O\big(\sqrt{n}\big)$ for $n \to +\infty$. It suffices to observe that
    \[
     \mathbb{P}\big(T(e,x)\geq \Hat{\upalpha}{\|x\|_1}\big) \leq \mathbb{P}\big(T(e,x)\geq \Hat{\upalpha}{\|x\|_1}\;|\;B\big) + \mathbb{P}(B^c)
    \]
    and we complete the proof for case $\|x\|_1=n$ combining (\ref{eq.lm4.1}) and (\ref{eq.lm4.2}), which ensures the existence of $\Hat{C}, \kappa>0$ satisfying
    \begin{equation} \label{first.all}
        \mathbb{P}\left(T(e,x)\geq \Hat{\upalpha}{\|x\|_1}\right) \leq  \Hat{C}\exp{\left(-\|x\|_1^\kappa\right)}.
    \end{equation}
    
    We now turn to the case $\|x\|_1<n$. Let $y \in \Gr$ be  such that $\|y\|_1 =n$ and $\|xy^{-1}\|_1 \geq n$. Observe that $\|xy^{-1}\|_1 < 2n$. Due to the subadditivity, we get $T(e,x) \leq T(e,y) + T(y,x)$ and by \eqref{first.all}, one has
    \begin{align*}
    \mathbb{P}\left(T(e,x)\geq 3\Hat{\upalpha}{n}\right) &\leq  \mathbb{P}\left(T(e,y)\geq \Hat{\upalpha}{n}\right) + \mathbb{P}\left(T(e,x)\geq \Hat{\upalpha}{\|xy^{-1}\|_1}\right)\\
    &\leq 2\Hat{C}\exp(-n^\kappa).
    \end{align*}
    Now, $C>0$ and $\upalpha>1$ can be conveniently chosen to arrive to the desired conclusion.
\end{proof}

\section{Asymptotic Shape} \label{sec.asymp.shape}

Before proving the asymptotic shape theorem, we present some basic concepts and results which will be useful for comparing $\Gr$ and $\Gr'$ (see \cite{cantrell2017,pansu1983} for further details).

\begin{dfntn}[Ergodic actions on probability spaces]
    Let $(\Gr,.)$ be a group and let $(\Omega, \mathscr{F}, \mathbb{P})$ be a probability space. A group action $\Gr \curvearrowright (\Omega,\mathscr{F},\mathbb{P})$ is said to be \emph{ergodic} if, given $A \in \mathscr{F}$ such that $xA = A ~\mbox{a.s.}$ for all $x \in \Gr$, then $\mathbb{P}(A) \in \{0,1\}$.
\end{dfntn}

\begin{dfntn}[Subadditive cocycle]
    Given a group $(\Gr,.)$ and a group action $\Gr \curvearrowright (\Omega,\mathscr{F},\mathbb{P})$, a function $c : \Gr \times \Omega \to \mathbb{R}^+$ is called a \emph{subaddittive cocycle} if
    \[
        c(xy, \omega) \leq c(y,\omega) + c(x,y\cdot\omega)
    \]
    for any $x,y \in \Gr$.
\end{dfntn}

 It is easy to verify that $T$ is subadditive, \emph{i.e.},
\[
    T(x,y) \leq T(x,z) + T(z,y)
\]
for any $x,y,z \in \Gr$, (see \cite[p.~538]{alves2002} where the authors proved the subadditivity on $\Z^D$). Now we can check that $c: \Gr \times \Omega \to \R_+$ with $c(x,\omega) := T(e,x)(\omega)$ is a subadditive cocycle considering the group action $\Gr \curvearrowright (\Omega,\mathscr{F}, \mathbb{P})$ given by $\pi_x(y\cdot \omega) = (x_0y^{-1}, (\xi_i)_{i\in\N^*})$ for any $x,y \in \Gr$ with $\pi_x(\omega) = (x_0,(\xi_i)_{i\in\N^*})$. This implies that $T(e,x)(y\cdot\omega) = T(y,xy)(\omega)$. 

Note that the group action defined above is ergodic since $\mathbb{P}$ is the product measure. Given $A \in \mathscr{F}$ such that $yA = A ~\mbox{a.s.}$ for all $y \in \Gr$, $\mathbb{P}_x(\pi_x[A]) = \mathbb{P}_y(\pi_y[A])$ for any $x,y \in \Gr$. Thus $\mathbb{P}(A) = \prod\limits_{x \in \Gr}\mathbb{P}_x(\pi_x[A]) \in \{0,1\}$.

To study the behavior of the asymptotic cone in the shape theorem, we follow the procedure adopted by \citet{pansu1983} and described by \citet[\S 2]{cantrell2017} in the construction of the limit space via quotient of $\Gr$ by the torsion subgroup of $N$ nilpotent. From now on we assume $\Gr' : = N / \tor N$ and $e' := \tor N$. We define the random variable $T'(x',y')$ by
\[
    T'(x',y') := \max\big\lbrace T(x,y) ~:~x.\tor N = x',~y.\tor N = y'  \big\rbrace
\]
for every $x', y' \in \Gr'$.

Consider the group action $\Gr' \curvearrowright (\Omega,\mathscr{F}, \mathbb{P})$ such that, for each $y .\tor N \in \Gr'$, we fix a ${y_0} \in y.\tor N$ and $\pi_x\big((y .\tor N)\cdot \omega\big) = (x_0y_0^{-1}, (\xi_i)_{i\in\N^*})$ where $\pi_x(\omega) = (x_0,(\xi_i)_{i\in\N^*})$.

\begin{lmm}
    The group action $\Gr' \curvearrowright (\Omega,\mathscr{F}, \mathbb{P})$ defined above is ergodic.
\end{lmm}
\begin{proof}
    Let $A \in \mathscr{F}$  be such that $A = x'A ~ \mbox{a.s.}$  for all $x' \in \Gr'$. Then $\prod\limits_{x \in y.\tor N}\mathbb{P}_x(\pi_x[A])$ assumes the same value for all $y.\tor N \in \Gr'$. Hence
    \[
        \mathbb{P}(A) = \prod_{y.\tor N \in \Gr'} \left( \prod\limits_{x \in y.\tor N}\mathbb{P}_x(\pi_x[A]) \right) \in \{0,1\},
    \]
    since $\Gr'$ defines a partition of $\Gr$ and $\mathbb{P}$ is a product probability measure.
\end{proof}

\begin{lmm}
    The function $c': \Gr' \hspace{-0.07cm}\times \Omega \to \R_+$ given by $c'(x',\omega) = T'(e',x')(\omega)$ is a subadditive cocycle.
\end{lmm}
\begin{proof}
    Let $z' \in \Gr'$ and $z \in z'$ such that $z'\cdot \omega = z\cdot\omega$. Then
    \begin{eqnarray*}
        T'(e',x'z')(\omega) &=& \max \{ T(y,xz)(\omega) : y \in \tor N,x \in x'\}\\
        &\leq& \max \{ T(u,z)(\omega) + T(v,x)(z\cdot \omega): u,v \in \tor N,x \in x'\}\\
        &\leq& T'(e',z')(\omega) + T'(e',x')(z'\cdot \omega)
    \end{eqnarray*}
    since the maximum function is subadditive
\end{proof}

\begin{lmm} \label{lm.T.bounded.torsion.as}
    Let $\varepsilon >0$ and let $r \geq 1$. Then there exists $n_0 \in \N$ such that for all $n \geq n_0$, for all $z_1,z_2 \in \tor N$, every $x_1 \in B(e,rn)$ and any $x_2 \in x_1.\tor N$, one has
    \[
        |T(z_1,x_1) - T(z_2,x_2)| < \varepsilon n \quad \mbox{a.s.}
    \]
    
\end{lmm}
\begin{proof}
    We first observe that it follows from the subadditivity of $T$ that
     \[
        |T(z_1,x_1) - T(z_2,x_2)| \leq \max\{T(x,yx): y \in \tor N, x \in x_1.\tor N\} +\max_{z,\tilde{z} \in \tor N} \{T(z,\tilde{z}),T(\tilde{z},z)\} .
     \]
     
     By Proposition \ref{prop.T.finite.as} and the finiteness of $\tor N$, there exist constants $C', \beta> 0$ such that 
     \[
        \mathbb{P}\left( T(e,y) \geq \varepsilon n : y \in \tor N \right) \leq C' e^{-\varepsilon^\beta n^\beta}.
     \]
     
     Since $T(e,y)$ and $T(x,yx)$ are identically distributed for every $x \in \Gr$, there exists a constant $c' >0$ such that
     \begin{eqnarray*}
        \mathbb{P}\left( \max\limits_{x \in B(e,rn)}\left\lbrace T(\tilde{x},y\tilde{x}): y \in \tor N, \tilde{x} \in x.\tor N \right\rbrace \geq \varepsilon n \right) &\leq& C'\frac{|B(e,rn)|}{e^{\varepsilon^\beta n^\beta}}\\
        &\leq& c'\frac{n^D}{e^{\varepsilon^\beta n^\beta}}.
     \end{eqnarray*}
     
     The desired result follows by application of Borel-Cantelli lemma.
\end{proof}

Recall that we are considering that $\mathcal{C}(\Gr,\Sgen)$ has polynomial growth and therefore $\Gr$ is virtually nilpotent. Then $N$ is a nilpotent normal subgroup and has finite index. Let $[\Gr:N] = m$ and fix $y_{1}, \dots, y_{m} \in \Gr$ as the representatives for each coset $N_{(i)} = y_{i}.N \in \Gr/N$. We will show below a result similar to Lemma \ref{lm.T.bounded.torsion.as} comparing $\Gr$ to $N$. 

\begin{lmm} \label{lm.T.nilpotent}
    Let $\varepsilon >0$ and let $r \geq 1$. Then there exists $n_0 \in \N$ such that for all $n \geq n_0$ and every $x \in B(e,rn)$, one has
    \[
        \left|T(e,x) - T\left(e,x_{(j)}\right)\right| < \varepsilon n \quad \mbox{a.s.}
    \]
    where $x \in N_{(j)}$ and $x_{(j)}:= y_{j}^{-1}.x \in N$.
\end{lmm}
\begin{proof}
     The proof follows by the same method used in Lemma \ref{lm.T.bounded.torsion.as}. Observe that 
     \[
        \left|T(e,x) - T\left(e,x_{(j)}\right)\right| \leq \max\left\{T\left(x,x_{(j)}\right),T\left(x_{(j)},x\right)\right\}.
     \]
     
     It follows from Proposition \ref{prop.T.finite.as} that there exists a constant $\Tilde{C} >0$
     \[
        \mathbb{P}\left( T(e,y) \geq \varepsilon n : y \in \left\{ y_
        {i}^{\pm 1}: i \in \{1,\dots, m\}\right\} \right) \leq \Tilde{C} e^{-\varepsilon^\beta n^\beta}.
     \]
     
     Since $T(e,y)$ and $T(z,yz)$ are identically distributed for every $z \in \Gr$, there exists a constant $\Tilde{c} >0$ such that
     \begin{eqnarray*}
        \mathbb{P}\left( \max\limits_{x \in B(e,rn)}\big\lbrace T\left(x,x_{(j)}\right), T\left(x_{(j)},x\right): x \in N_{(j)}\in\Gr/N \big\rbrace \geq \varepsilon n \right) &\leq& \Tilde{c}\frac{n^D}{e^{\varepsilon^\beta n^\beta}}.
     \end{eqnarray*}
     
     We now apply the Borel-Cantelli lemma completing the proof.
\end{proof}

As a consequence of the lemma above, $T(e,x)$ and $T(e,x_{(j)})$ are asymptotically equivalent $a.s.$ as $\|x\|_1\to \infty$. Therefore some results on the asymptotic behaviour for nilpotent groups may be extended to $T$ on $\Gr$ $a.s.$ We state below, without proof, a proposition from  \citet[p.~128]{austin2016} and improved by \citet[p.~1328]{cantrell2017}.

\begin{prpstn} \label{ergodic.cocicle.nilpotent.thm}
    Let $\Lambda$ be a finitely generated torsion-free nilpotent group. Consider a subadditive cocycle $ \mathbf{c}: \Lambda \hspace{-0.09cm}\times \Omega \to \R_+$ such that $\mathbf{c}({\lambda}, -) \in L^1(\Omega, \mathscr{F}, \mathbb{P})$ for all $\lambda \in \Lambda$ associated to an ergodic group action $\Lambda \curvearrowright (\Omega, \mathscr{F}, \mathbb{P})$. Then there exists a set $\Omega' \subseteq \Omega$ such that $\mathbb{P}(\Omega')=1$ and, for a given $\omega \in \Omega'$ and $x \in \Lambda$,
    \[
        \lim_{n\to\infty} \frac{1}{n}\mathbf{c}\big(\lambda^n,\omega\big) = \phi(\lambda^{\text{ab}}),
    \]
    where $\phi:\Lambda^{\text{ab}}\otimes \R \to \R_+$ is homogeneous, subadditive and uniquely defined on its domain. Here $\Lambda^{\text{ab}} := \Lambda/[\Lambda,\Lambda]$ is the abelianization of $\Lambda$ and $\lambda^{\text{ab}} := \lambda[\Lambda,\Lambda]$. In particular,
    \[
        \phi(\lambda^{\text{ab}}) = \inf_{n \geq 1} \left\lbrace \frac{1}{n}\mathbb{E}\big[\mathbf{c}\big({\lambda_1}^n,-\big)\big] : {\lambda}^{\text{ab}} = \lambda_1^{\text{ab}} \right\rbrace.
    \]
\end{prpstn}

We denote by $\mathfrak{g}_\infty$ the Lie algebra associated with $G_\infty$. The algebra $\mathfrak{g}_\infty$ can be defined following the construction of the asymptotic cone from \citet{pansu1983} (also found in \cite{cantrell2017}). The limit space $G_\infty$ is also known as a Carnot group and we may write $\mathfrak{g}_\infty = \bigoplus_{i=1}^k\mathfrak{v}_i$. The homotheties $\delta_t$ from Theorem \ref{shape.thm} and Pansu's theorem are linear endomorphims of $\mathfrak{g}_\infty$ given by $\delta_t(v_j) = t^jv_j$ for $v_j\in \mathfrak{v_j}$ (see \cite{breuillard2014} for more details).

We call $\mathfrak{v}_1$ horizontal space. By abuse of notation, we may write $\delta_t(g) $ for $\exp_\infty(\delta_t(\log_\infty(g)))$, where $\exp_\infty$ is the exponential map and $g \in G_\infty$. As one may verify from the Baker-Campbell-Hausdorff formula, $\delta_t(g)\delta_t(g')=\delta_t(g.g')$. Also, for simplicity, we write $\phi((x')^{\text{ab}})$ for $x'\in \Gr'$, since $\mathfrak{g}_\infty^{\text{ab}} \cong \mathfrak{v}_1 \cong (\Gr')^{\text{ab}}\otimes\R$.

The corollary below follows directly from Proposition \ref{ergodic.cocicle.nilpotent.thm}, Lemma \ref{lm.T.nilpotent} and the integrability of $T$ given by Proposition \ref{prop.T.finite.as}, since $\Gr'$ is torsion-free and it suffices to verify the asymptotic behaviour of $T$ on $N$. 

\begin{crllr} \label{cor.subadd.cocycle}
    Let $x' \in \Gr'$, then there exists an unique homogeneous subadditive function $\phi: \mathfrak{g}_\infty^{\text{ab}} \to \R_+$ such that
    \[
        \lim_{n\to\infty}\frac{1}{n}T'(e',(x')^n) = \phi((x')^\text{ab}) \quad \mbox{a.s.}
    \]
    In particular, we get
    \[
        \phi((x')^\text{ab}) = \inf_{n \geq 1} \left\{ \frac{1}{n}\mathbb{E}\big[T'\big(e',{(z')}^n\big)\big]: z' \in x'[\Gr',\Gr']\right\}.
    \]
\end{crllr}

The abelian case is rather simple and will be used in the proof of the shape theorem. First, one can define $G_\infty$ to be such that $\Gr' \leq G_\infty$, where $\Gr'$ is a lattice in $G_\infty$. We simply have $\g_\infty \cong \g_\infty^{\text{ab}}\cong \Z^D \otimes \R \cong \R^D$ with the null bracket, then it is an abelian Lie Algebra. Futhermore, $\delta_t(v)=tv$ and $\delta_n(x')=(x')^n$ for $v \in \g_\infty$ and $x' \in \Gr'$. Now $d_\phi$ is induced by the quasinorm $\phi$ in the commutative case.

Although our main result considers only abelian groups, all the intermediate results were proved for $\Gr$ with polynomial growth rate $D\geq3$. Even though $\Gr' \cong \Z^D$,  it is not necessarily the case that the corresponding graph is isomorphic to the hypercubic $\Z^D$ lattice. Hence, we show that the torsion subgroup does not interfere in the limit shape and that this limit behavior does exist for every finite symmetric generator set $\Sgen$.

We now turn to the proof of the shape theorem. Let us first emphasize that the generalization for the case where $\mathcal{C}(\Gr,\Sgen)$ has polynomial growth (therefore $\Gr$ is virtually nilpotent) is not trivial. It would be still necessary to check how $\frac{1}{n}d_\omega$ converges on the horizontal subspaces of $G_\infty$, since it is a sub-Riemannian manifold when associated to $d_\infty$. In that case $\phi$ is not defined on $\g_\infty$ and $d_\phi$ is defined by admissible curves in $G_\infty$. 
     
Recent works on the shape theorem for the first passage percolation model considered the case where $\Gr$ is not necessarily abelian, namely -- Benjamini and Tessera in \cite{benjamini2015}, and, Cantrell and Furman in \cite{cantrell2017}. However, the frog model does not satisfy the hypotheses under which the results were shown. Benjamini and Tessera considered that the weight of each edge on the graph is given by i.i.d. random variables with exponential moment. Cantrell and Furman studied the case where $c(x,\omega)$ is bi-Lipschitz with respect to the word norm and the cocycle is conditioned to an additional innernerss assumption.

\begin{proof}[Proof of Theorem \ref{shape.thm}]
     We begin by observing that by Lemma \ref{lm.T.bounded.torsion.as} $T(e, x)$ is asymptotically equivalent to $T'(e',x.\tor\Gr)$ $a.s.$ as $\|x\|_1 \to \infty$. It follows from the definition of the model that $\|x\|_1 \leq T(e,x)$. The lower bound corresponds to the case in which the particle follows a $d$-geodesic on $\Gr$. Thus $B_\omega(e,n) \subseteq B(e,n)$. By an application of Lemma \ref{lm.T.bounded.torsion.as}, the Hausdorff distance between $\Gr$ and $\Gr'$ with respect to $\frac{1}{n}d_\omega$ converges to zero a.s. Then it suffices to prove the shape theorem for $T'$ on $\Gr'$ to get the desired conclusion.

     Observe that $\Sgen' = \{ s.\tor\Gr ~|~ s \in \Sgen \}$ is a finite symmetric generator set of $\Gr'$. We get from \eqref{structural.abelian.thm} and \S\ref{conv.metric.spaces} that $\Gr' \cong \Z^D$ and the limit space can be defined in such a way that $\Gr' \leq G_\infty$ where $G_\infty \cong \R^D$. Consider the norm $\|-\|_1'$ on the quotient space $\Gr'$ given by
     \[
        \|x'\|_1' = \inf \left\lbrace \|x\| ~:~ x.\tor\Gr = x' \right\rbrace.
     \]
     Under the conditions stated above, $G_\infty$ is equipped with a Riemannian metric $d_\infty$ associated with the rescaled metric $\frac{1}{n}d'$ of $\Gr'$ induced by $\|-\|_1'$.
     
    Let $\phi:\mathfrak{g}_\infty\to \R_+$ be given by Corollary \ref{cor.subadd.cocycle} and note that $\mathfrak{g}_\infty^{\text{ab}}\cong \Gr'\otimes\R\cong\R^D$. It is easily seen  that $\|x'\|_1' \leq \phi(x')$. Since $\phi$ is subadditive and $\Gr'$ is finitely generated, $\phi$ is Lipschitz. Thus we can apply the construction described in \cite[\S2.2-2.3]{cantrell2017} to verify that $d_\phi$ is bi-Lipschitz with respect to $d_\infty$.

    We now show that
    \begin{equation} \label{T.asymp.phi}
        \lim_{\|x'\|_1'\to +\infty} \frac{T'(e',x') - \phi(x')}{\|x'\|_1'} = 0 \quad a.s.,
    \end{equation}
    which is a rather standard technique to verify the abelian case.
    
    Let $y_n' \in \Gr'$ be such that $\|y_n'\|_1' \to +\infty$. We write $t \loz y'$ for $\delta_t(y')$ where $y' \in \Gr' \leq G_\infty$ and $t >0$. Since $\frac{1}{\|y_n\|_1'} \loz y_n$ is bounded, we can extract a convergent subsequence and there exists $g \in G_\infty$ such that
    \[
        \lim\limits_{n\to +\infty}d_\infty\left(\frac{1}{\|y_n\|_1'}\loz y_n,g\right)= 0.
    \]
    Let $\varepsilon>0$. Since $G_\infty$ is the asymptotic cone obtained by the rescaling $\frac{1}{n}d'$, there exist $z' \in \Gr'$ and $m' \in \N$ such that
    \[
        d_\infty\left(\frac{1}{m'}\loz z',g\right) < \varepsilon.
    \]
    Fix $h_n:= \left\lfloor \frac{\|y_n'\|_1'}{m'} \right\rfloor$, then one can easily see that
    \begin{align}
        d_\infty\left(y_n, h_n\loz z'\right) &\leq d_\infty\left(y_n,\frac{\|y_n'\|_1'}{m'}\loz z'\right)+\left| \frac{\|y_n'\|'_1}{m'} - h_n\right| d_\infty\left(e,z'\right) \nonumber \\
        &\leq \|y_n\|_1' d_\infty\left(\frac{1}{\|y_n'\|_1'}\loz y_n, \frac{1}{m'}\loz z'\right) + \|z'\|_1'. \label{ineq.hn}
    \end{align}
    
    Set $n_0 \in \N$ to be such that, for all $n \geq n_0$, $d_\infty\left( \frac{1}{\|y_n'\|_1'}\loz y_n',g \right)< \varepsilon$. Hence it follows from \eqref{ineq.hn} and $\|y_n\|_1' \to \infty$ that, for sufficiently large $n$,
    \begin{equation} \label{bound.ynzn}
        d_\infty(y_n',h_n\loz z') \leq 3 \varepsilon \|y_n\|_1'.
    \end{equation}
    Let us now write $z_n' := h_n\loz z' = (z')^{h_n} \in \Gr'$. Observe that it follows from subadditivity that
    \begin{equation} \label{bound.Tynzn}
        |T'(e',y_n')-T'(e',z_n')| \leq \max\{T'(y_n',z_n'), T'(z_n',y_n')\}.
    \end{equation}
    By Proposition \ref{prop.T.div.dist}, there exists $\Check{C}>0$ such that
    \begin{equation} \label{prob.T.final}
        \mathbb{P}\left(\sup_{\|x\|_1 \leq n} \left\{T(x_n,x ) \;:\; x_n \in y_n' \cup z_n' \right\} \geq \upalpha n\right) \leq \Check{C} \frac{n^D}{e^{n^{\kappa}}}.
    \end{equation}
    We apply Borel-Cantelli lemma to \eqref{prob.T.final} and we verify by \eqref{bound.ynzn} and \eqref{bound.Tynzn} that
    \begin{equation} \label{T.bound}
        |T'(e',y_n')-T'(e',z_n')| \leq 3\upalpha' \varepsilon \|y_n'\|_1' \quad a.s., \quad \text{for}~n \gg 1.
    \end{equation}
    Finally, by Corollary \ref{cor.subadd.cocycle}, \eqref{T.bound}, and the properties of $\phi$, there exist $\Hat{k}, K>0$ such that
    \begin{align*}
        &\begin{aligned}
            |T'(e',y_n') - \phi(y_n')| &\leq |T'(e',y_n')-T'(e',z_n')| + |T'(e',z_n') - \phi(h_n \loz z')|+ |\phi(h_n\loz z')- \phi(y_n)|\end{aligned}\\
        &\begin{aligned}
            \phantom{|T'(e',y_n') - \phi(y_n')|} &\leq 3\alpha'\varepsilon \|y_n'\|_1' + h_n\varepsilon + 3\Hat{k} \varepsilon\|y_n\|_1' \leq K \|y_n'\|_1' \varepsilon \quad\quad a.s.
        \end{aligned}
    \end{align*}
     for sufficiently large $n$, which proves \eqref{T.asymp.phi}. Now, we conclude that $T(e,x)$ is almost surely asymptotically equivalent to $\phi(x.\tor\Gr)$ and the proof is complete.
\end{proof}

\begin{acknowledgement}
\textit{Acknowledgements:}
We would like to thank the anonymous reviewers for their careful reading of the manuscript and the comments that eventually led to an improved presentation.
\end{acknowledgement}

\bibliographystyle{abbrvnat}
\bibliography{references}
\end{document}